\documentclass[12pt,reqno]{amsart}

\usepackage{mathrsfs}
\usepackage{amsfonts}
\usepackage{amssymb}
\usepackage{amsxtra}
\usepackage{dsfont}
\usepackage{color}
\usepackage{graphicx}
\usepackage[english,polish]{babel}
\usepackage{enumerate}

\usepackage[compress, sort]{cite}

\usepackage{tikz}
\usetikzlibrary{calc}

\usepackage{newtxmath}
\usepackage{newtxtext}

\usepackage[margin=2.5cm, centering]{geometry}
\usepackage[colorlinks,citecolor=blue,urlcolor=blue,bookmarks=true]{hyperref}

\newcommand{\CC}{\mathbb{C}}

\newcommand{\NN}{\mathbb{N}}
\newcommand{\RR}{\mathbb{R}}
\newcommand{\PP}{\mathbb{P}}

\newcommand{\calC}{\mathcal{C}}
\newcommand{\calP}{\mathcal{P}}
\newcommand{\calQ}{\mathcal{Q}}

\newcommand{\Hes}{\mathcal{H}}

\newcommand{\Id}{\operatorname{Id}}

\newcommand{\pl}[1]{\foreignlanguage{polish}{#1}}

\newcommand{\tr}{\operatorname{tr}}

\newcommand{\Dom}{\operatorname{Dom}}

\newcommand{\lin}{\operatorname{span}}

\newcommand{\ud}{{\: \rm d}}

\newcommand{\supp}{\operatornamewithlimits{supp}}

\newcommand{\Res}{\operatorname{Res}}

\newtheorem{theorem}{Theorem}[section]

\newtheorem{proposition}[theorem]{Proposition}
\newtheorem{lemma}[theorem]{Lemma}
\newtheorem{corollary}[theorem]{Corollary}

\newtheorem*{theorem*}{Theorem}

\theoremstyle{definition}

\newtheorem{problem}[theorem]{Problem}

\numberwithin{equation}{section}
\numberwithin{theorem}{section}

\usepackage[utf8]{inputenc}

\title{Christoffel functions for multiple orthogonal polynomials}

\date{\today}

\author{Grzegorz \'Swiderski}
\address{
	Grzegorz \'Swiderski \\
	Department of Mathematics \\
	KU Leuven \\
	Celestijnenlaan 200B box 2400 \\
	BE-3001 Leuven \\
	Belgium \&
	University of Wroc\l{}aw \\
	Faculty of Mathematics and Computer Science \\
	pl. Grunwaldzki 2/4 \\
	50-384 Wroc\l{}aw \\
	Poland
}
\email{grzegorz.swiderski@kuleuven.be}

\author{Walter Van Assche}
\address{
	\pl{
	Walter Van Assche \\
	Department of Mathematics \\
	KU Leuven \\
	Celestijnenlaan 200B box 2400 \\
	BE-3001 Leuven \\
	Belgium}
}
\email{walter.vanassche@kuleuven.be}

\subjclass[2020]{33C45; 41A28; 42C05}
\keywords{Multiple orthogonal polynomials, Christoffel--Darboux kernel, zero distribution, Nevai's operators}

\begin{document}
\selectlanguage{english}
\maketitle

\begin{abstract}
We study weak asymptotic behaviour of the Christoffel--Darboux kernel on the main diagonal corresponding to multiple orthogonal polynomials. We show that under some hypotheses the weak limit of $\tfrac{1}{n} K_n(x,x) \ud \mu$ is the same as the limit of the normalized zero counting measure of type II multiple orthogonal polynomials. We also study an extension of Nevai's operators to our context.
\end{abstract}

\section{Introduction} \label{sec:1}
Let $\{p_n \}_{n=0}^\infty$ be the orthonormal polynomials for a positive
measure $\mu$ on the real line,
\[
	\int_{\mathbb{R}} p_n(x)p_m(x) \ud \mu(x) = \delta_{m,n}, 
\]
then the Christoffel--Darboux kernel is given by
\[
	K_n(x,y) = \sum_{k=0}^{n-1} p_k(x)p_k(y), 
\]
and the Christoffel function is
\[
	\lambda_n(x) = \frac{1}{K_n(x,x)}.  
\]
The Christoffel--Darboux kernel and Christoffel function play an important role in the theory of orthogonal polynomials, polynomial least squares approximation, 
the moment problem, approximation of weight functions, and universality in random matrix theory, see, e.g., the long survey of Nevai \cite{Nevai1986}, and the papers of M\'at\'e-Nevai-Totik \cite{Mate1991}, Van Assche \cite{VanAssche1993}, Inglese \cite{Inglese1995}, Totik \cite{Totik2000, Totik2016}, Simon \cite{Simon2008}, and Lubinsky \cite{Lubinsky2009, Lubinsky2011}. 
The Christoffel--Darboux kernel can be expressed in terms of the polynomials $p_n$ and $p_{n-1}$ through the Christoffel--Darboux formula
\[    
	K_n(x,y) = a_n \frac{p_n(x)p_{n-1}(y)-p_{n-1}(x)p_n(y)}{x-y}, 
\]
where $a_n$ is one of the coefficients of the three-term recurrence relation for the orthonormal polynomials. The Christoffel function
is a positive function on the real line and satisfies an extremum problem
\[
	\lambda_n(x) = 
	\min_{\substack{p \in \mathbb{P}_{n-1}\\p(x)=1}} 
	\int_{\mathbb{R}} |p(y)|^2 \ud \mu(y),
\]
where $\PP_n \subset \RR[x]$ is the space of real polynomials with degree less than or equal to $n$.
Furthermore it is clear that the Christoffel--Darboux kernel $K_n(x,y)$ is symmetric in the two variables $x,y$.
   
In this paper we will consider the Christoffel--Darboux kernel for multiple orthogonal polynomials. In this case the symmetry is lost, there is no obvious
extremum problem, and the positivity of the Christoffel function is not immediately visible since it is no longer a sum of squares. Nevertheless we
will be able to give some results about the weak convergence of the Christoffel--Darboux kernel.

Let $r \geq 1$ and let $\vec{\mu} = (\mu_1,\ldots,\mu_r)$ be a vector of positive measures on the real line having all moments finite. 
By $\NN$ we denote the set of positive integers and $\NN_0 = \NN \cup \{0\}$.
Let $\vec{n} \in \NN_0^r$ be a multi-index of size $|\vec{n}| = n_1+\ldots+n_r$. The monic polynomial $P_{\vec{n}} \in \RR[x]$ is called the \emph{type II multiple orthogonal polynomial} if it of degree $|\vec{n}|$ and it satisfies the following simultaneous orthogonality
\begin{equation} \label{eq:46}
   \int_\RR x^k P_{\vec{n}}(x) \ud \mu_j(x) = 0, \qquad  
   0 \leq k \leq n_j-1; 1 \leq j \leq r. 
\end{equation}
The existence of $P_{\vec{n}}$ is not automatic but it holds under some additional conditions imposed on the moments of the measures. If for any $\vec{n} \in \NN_0^r$ the polynomial $P_{\vec{n}}$ exists, then the vector $\vec{\mu}$ is called \emph{perfect}. The class of perfect systems contains: Angelesco, Nikishin or more generally AT systems (see \cite[Chapter 23]{Ismail2009} for more details). In this article we shall assume that $\vec{\mu}$ is perfect.

Multiple orthogonal polynomials have applications in such fields as approximation theory (Hermite-P\'ade approximation \cite{VanAssche2006}, construction of quadratures \cite{Coussement2005, LubWVA, WVAVuer}), number theory (proving irrationality of numbers \cite{VanAssche1999}), random matrix theory (models with external source \cite{Bleher2004} and products of random matrices \cite{KuijlZhang, KieKuijlStiv}) and more general determinantal point processes (see, e.g., \cite{Kuijlaars2010}).

In the applications to determinantal point processes one is interested in the asymptotic behaviour of the corresponding Christoffel--Darboux kernel which is the main object of study in the present paper. In order to define it we need some definitions. First of all, we need a dual concept to \eqref{eq:46}. A vector  $A_{\vec{n}} = (A_{\vec{n},1},\ldots,A_{\vec{n},r})$ contains \emph{type I multiple orthogonal polynomials} if each $A_{\vec{n},j} \in \RR[x]$ is a polynomial of degree $\leq n_j-1$ and
\[
    \sum_{j=1}^r \int_\RR x^k A_{\vec{n},j}(x) \ud \mu_j(x) = 0, \qquad 0 \leq k \leq |\vec{n}|-2, 
\]
with the normalization
\[
    \sum_{j=1}^r \int_\RR x^{|\vec{n}|-1} A_{\vec{n},j}(x) \ud \mu_j(x) = 1.  
\]
It is a basic result that $A_{\vec{n}}$ exists if and only if $P_{\vec{n}}$ exists.
Next, without loss of generality we can assume that the measures $\mu_1,\ldots,\mu_r$ are absolutely continuous with respect to a measure $\mu$ (e.g., one can take $\mu = \mu_1+\mu_2 + \cdots +\mu_r$). Let $w_j$ be the Radon-Nikodym derivative of $\mu_j$ with respect to $\mu$. Then one defines the function
\[
	Q_{\vec{n}}(x) = \sum_{j=1}^r A_{\vec{n},j}(x) w_j(x).
\]
Further, let us fix a sequence of multi-indices from $\NN_0^r$ such that for any $\ell \in \NN_0$
\begin{equation} \label{eq:24}
	|\vec{n}_\ell| = \ell \quad \text{and} \quad 
	\vec{n}_{\ell+1} = \vec{n}_\ell + \vec{e}_{i_\ell}
\end{equation}
for some $i_\ell \in \{1,\ldots,r\}$, where $\vec{e}_j \in \NN_0^r$ is equal to $1$ on the $j$th position and $0$ elsewhere. Next, let us define
\begin{equation} \label{eq:25}
	p_{\ell} := P_{\vec{n}_{\ell}}, \quad 
	q_{\ell} := Q_{\vec{n}_{\ell+1}}.
\end{equation}
Then the sequences $(p_{\ell} : \ell \in \NN_0), (q_{\ell} : \ell \in \NN_0)$ are biorthogonal in $L^2(\mu)$, \cite[\S 23.1.3]{Ismail2009}, i.e.
\begin{equation} \label{eq:1}
	\int_\RR p_{\ell}(x) {q_{\ell'}(x)} \ud \mu(x) = 
	\begin{cases}
		1 & \ell = \ell' \\
		0 & \text{otherwise}.
	\end{cases}
\end{equation}
Finally, one defines the \emph{Christoffel--Darboux kernel} by the formula
\begin{equation} \label{eq:47}
	K_n(x,y) = \sum_{j=0}^{n-1} p_j(x) {q_j(y)}. 
\end{equation}

Observe that the kernel $K_n(x,y)$ is usually non-symmetric, i.e., $K_n(x,y) \neq K_n(y,x)$, unless $r=1$ and $\mu:=\mu_1$. This kernel has been studied extensively in the case $r=1$ (and $\mu=\mu_1$) for compactly supported measures $\mu$, see, e.g., the survey \cite{Nevai1986} or \cite{Simon2008} for details. In particular, for any $n \in \NN$ and $x \in \supp(\mu)$ one has $K_n(x,x) \geq 1$, and
\begin{equation} \label{eq:52}
	\lim_{n \to \infty} K_n(x,x) = \frac{1}{\mu(\{x\})},
\end{equation}
and for any $f \in \calC_b(\RR)$ (continuous bounded function on the real line)
\begin{equation} \label{eq:48}
	\lim_{n \to \infty} \bigg| \int_\RR f \ud \nu_n - \int_\RR f \ud \eta_n \bigg| = 0,
\end{equation}
where
\begin{equation} \label{eq:49}
	\nu_n = \frac{1}{n} \sum_{x \in p_n^{-1} (\{0\})} \delta_x \quad \text{and} \quad
	\ud \eta_n(x) = \frac{1}{n} K_n(x,x) \ud \mu(x).
\end{equation}
The measure $\nu_n$ is the normalized zero counting measure of $p_n$. The relation \eqref{eq:48} tells us that the weak accumulation points of the sequences $(\nu_n : n \in \NN)$ and $(\eta_n : n \in \NN)$ (by weak compactness they exist) are the same. This property is very important because the behaviour of the sequence $(\nu_n : n \in \NN)$ is rather well-understood in terms of logarithmic potential theory (see, e.g., \cite{Simon2007}). Let us mention that in the applications to random matrix theory one is usually interested in stronger pointwise limits
\[
	\lim_{n \to \infty} \frac{1}{n^\gamma} K_n \bigg(x + \frac{a}{n^\gamma}, x + \frac{b}{n^\gamma} \bigg)
\]
for $a,b \in \RR$ and some $\gamma > 0$, see e.g. \cite{Totik2009, Lubinsky2016}, which are out of the scope of the present article.

The idea of the proof of \eqref{eq:48} presented in \cite{Simon2009} is to show that \eqref{eq:48} holds for any $f \in \RR[x]$ by expressing both integrals in terms of the corresponding Jacobi matrix. A crucial feature is that in this setup the Jacobi matrix is a bounded operator on $\ell^2$. Since $(\nu_n : n \in \NN)$ and $(\eta_n : n \in \NN)$ are sequences of probability measures, then \eqref{eq:48} holds for any $f \in \calC_b(\RR)$ by a density argument.

In Theorem~\ref{thm:A} below we adapt this approach to $r > 1$ to obtain that under some hypotheses \eqref{eq:48} holds true for any $f \in \RR[x]$. To do so, we need to define a generalisation of the Jacobi matrix to our setup. More precisely, since the sequence $(p_\ell : \ell \geq 0)$ is an algebraic basis for $\RR[x]$ there are real constants $J_{\ell,k}$ such that
\begin{equation} \label{eq:53}
	x p_\ell = \sum_{k=0}^\infty J_{\ell,k} p_k.
\end{equation}
We collect these constants into a matrix $J = [J_{\ell,k}]_{\ell,k=0,1,\ldots}$, which is lower Hessenberg, i.e. $J_{\ell,k} = 0$ for $k-\ell > 1$. Let us remark that for a specific choice of the sequence $(\vec{n}_\ell: \ell \geq 0)$, namely the so-called stepline multi-indices, the matrix $J$ is both banded and bounded (see \cite{Aptekarev2006}). This is not true anymore in our generality hence it makes our analysis more complicated. The idea of using $J$ in the general setup comes from \cite{Duits2021}.
Hardy \cite{Hardy2015,Hardy2018} proved results comparing the distribution of the zeros of $p_n$ with the
distribution of random points of special determinantal point processes known as polynomial ensembles. He used very similar techniques which he calls tracial representations (see Lemma 2.1 and Corollary 2.4 in \cite{Hardy2015}) and his results about the average empirical
distribution of the point process \cite[Thm. 3.1]{Hardy2018} corresponds to \eqref{eq:48} when $f$ is a polynomial.  

\begin{theorem} \label{thm:A}
Let $(\vec{n}_\ell : \ell \geq 0)$ be a sequence satisfying \eqref{eq:24}. Assume that the corresponding matrix $J$ satisfies
\begin{equation} \label{eq:44}
	\sup_{n \geq N} |J_{n,n-N}| < \infty \quad \text{for any } N \geq 0.
\end{equation}
Then for any $f \in \RR[x]$ the formula \eqref{eq:48} holds true.
\end{theorem}
It turns out that the condition \eqref{eq:44} is the right substitute for the boundedness of the Jacobi matrix. In Theorem~\ref{thm:3} we formulate a sufficient condition for \eqref{eq:44} to hold which is satisfied for Angelesco and AT systems. Moreover, in Lemma~\ref{lem:2} we present a sufficient condition for the positivity of $K_n(x,x)$, which allows us to prove that \eqref{eq:48} holds for any $f \in \calC_b(\RR)$. This condition covers compactly supported Angelesco systems and compactly supported AT systems with continuous densities.

Inspired by \cite[Section 6.2]{Nevai1979} we define for every bounded measurable function $f$ and $x \in \supp(\mu)$
\[
	G_n[f](x) = \frac{1}{K_n(x,x)} \int_\RR K_n(x,y) K_n(y,x) f(y) \ud \mu(y) ,
\]
provided that $K_n(x,x) \neq 0$. 
We are interested in those cases when 
\begin{equation} \label{eq:50}
	\lim_{n \to \infty}
	G_n[f](x) = 
	f(x), \quad f \in \calC_b(\RR) .
\end{equation}
In the classical case $r=1$ the condition \eqref{eq:50} has been introduced in \cite{Nevai1979} in order to obtain relative asymptotics of Christoffel functions (see, e.g., \cite[Section 4.5]{Nevai1986}). It has been used also for pointwise convergence of orthogonal expansions with respect to systems of orthonormal polynomials (see \cite[Section 4.12]{Nevai1986}). It has been studied rather extensively in \cite{Breuer2010a} where it was shown that for compactly supported measures $\mu:= \mu_1$ the condition \eqref{eq:50} is equivalent to subexponential growth of the sequence of orthonormal polynomials $(p_n : n \geq 0)$ in $L^2(\mu)$, namely
\begin{equation} \label{eq:51}
	\lim_{n \to \infty} \frac{p_n^2(x)}{\sum_{j=0}^n p_j^2(x)} = 0.
\end{equation}
In \cite{Lubinsky2011}  the problem of convergence of $G_n[f]$ to $f$ in the $L^p(\!\ud x)$ norm was considered. Finally, in \cite{Breuer2014} this concept has been applied to estimation of the variance of linear statistics coming from orthogonal polynomial ensembles.
Our motivation of examining the condition \eqref{eq:50} comes from the desire to understand whether an analogue of \eqref{eq:52} holds true (see Proposition~\ref{prop:6}).

\begin{theorem} \label{thm:B}
Assume that:
\begin{enumerate}[(a)]
\item the matrix $J$ satisfies \eqref{eq:44},
\item $K_n(x,x) > 0$ for large $n$,
\item for any $N \geq 1$ one has
$\begin{aligned}[b]
	\lim_{n \to \infty} \frac{q_{\ell}(x) p_{\ell'}(x)}{K_n(x,x)} = 0
\end{aligned}$, where $\ell' \in [n-N, n)$ and $\ell \in [n, n+N]$. \label{thm:B:d}
\end{enumerate}
Then the convergence \eqref{eq:50} holds for any $f \in \RR[x]$.
\end{theorem}
Let us observe that \eqref{thm:B:d} is an analogue of \eqref{eq:51}.
Regarding Theorem~\ref{thm:B}, we were not able to extend the convergence \eqref{eq:50} to any $f \in \calC_b(\RR)$. The problem here is that the kernel $K_n(x,y) K_n(y,x)$ might not be of constant sign --- we observed this phenomenon numerically for Jacobi-Pi\~{n}eiro polynomials. 

We prove Theorem~\ref{thm:B} by similar means as for Theorem~\ref{thm:A}. Namely, we were able to rewrite $G_n[f](x)$ in terms of the matrix $J$ (see Lemma~\ref{lem:1}), which seems to be a novel approach to this problem even for $r=1$.

The article is organized as follows. In Section~\ref{sec:4} we define and analyse the matrix $J$ from \eqref{eq:53} and its finite truncations. Section~\ref{sec:5} is devoted to analysing the condition~\eqref{eq:44}, its consequences and sufficient conditions. In Section~\ref{sec:6} we prove Theorem~\ref{thm:A}, its corollaries and we discuss conditions under which we can extend convergence to $\calC_b(\RR)$. In Section~\ref{sec:7} we prove Theorem~\ref{thm:B} and discuss some open problems related to it. Finally, in Section~\ref{sec:8} we discuss applications of Theorem~\ref{thm:A} to classes of multiple orthogonal polynomials when one can describe the limit of $(\nu_n : n \in \NN)$ more explicitly.

\section{The matrix representations of the multiplication operator} \label{sec:4}

Consider the multiplication operator $M_x : \Dom(M_x) \to L^2(\mu)$ given by
\[
	(M_x f)(x) := x f(x),
\]
where 
\[
	\Dom(M_x) = \big\{ f \in L^2(\mu) : x \cdot f \in L^2(\mu) \big\}.
\]
We have that $M_x$ is a (possibly unbounded) self-adjoint operator.

Let us define the linear spaces
\[
	\calP_n := \big\{ p \in \RR[x] : \deg(p) \leq n \big\}, \quad n \geq 0.
\]

We would like to examine the matrix representation of $M_x$ on the space $\RR[x]$. Since
$(p_\ell : \ell \in \NN_0)$ is an algebraic basis for the space $\RR[x]$, it is enough 
to examine the action of $M_x$ on each $p_{\ell}$. Because $x p_{\ell} \in \calP_{\ell+1}$
there are constants $\{ J_{\ell, k} \}_{k=0}^{\ell+1}$ such that
\begin{equation} \label{eq:2a}
	M_x p_{\ell} = \sum_{k=0}^{\ell+1} J_{\ell, k} p_{k} = \sum_{k=0}^{\infty} J_{\ell, k} p_{k},
\end{equation}
where we have defined $J_{\ell, k} = 0$ for $k > \ell+1$.
Let us collect these constants into a matrix $J = [J_{\ell, k}]_{\ell,k=0}^\infty$.
Observe that it is a lower Hessenberg matrix. 

Next, we want to examine the spectrum of finite sections of $J$.
\begin{proposition} \label{prop:2}
For any $n \geq 1$ let
\[
	J_n := [J]_{0 \ldots n-1;0 \ldots n-1}.
\]
Then $\det(x \Id_n - J_n) = p_n(x)$. In particular, $\sigma(J_n) = p_{n}^{-1}[\{0\}]$, i.e., the eigenvalues of $J_n$ are equal to the
zeros of $p_n$.
\end{proposition}
\begin{proof}
The proof is analogous to \cite[Section 2.2]{Coussement2005}. Namely, observe that
\[
	x \Id_n - J_n =
	\begin{pmatrix}
		x-J_{0,0} & -1 & 0   & 0  & \ldots & 0 &  0   \\
		-J_{1,0} & x-J_{1,1} & -1 & 0 & \ldots & 0 & 0   \\
		-J_{2,0} & -J_{2,1} & x-J_{2,2} & -1 & \ldots &0 & 0   \\
		 \ldots & \ldots & \ldots & \ldots & \ldots & \ldots & \ldots\\
		-J_{n-2,0} & -J_{n-2,1} & -J_{n-2,2}  & \ldots & \ldots & x-J_{n-2,n-2} & -1 \\
		-J_{n-1,0} & -J_{n-1,1} & -J_{n-1,2} & \ldots & \ldots &-J_{n-1,n-2} & x-J_{n-1,n-1} \\
	\end{pmatrix}.
\]
Set
\[
	f_n(x) = \det(x \Id_n - J_n), \quad n \geq 0,
\]
where the determinant of an empty matrix is defined to be $1$. Then by expanding this determinant with respect to the last row and repeatedly expanding the resulting determinants with respect to the last columns we get
\[
	f_n(x) = (x-J_{n-1,n-1})f_{n-1}(x) -\sum_{j=1}^{n-1} (-1)^{n+j} J_{n-1,j-1} (-1)^{n-j} f_{j-1}(x),
\]
which results in
\begin{equation} \label{eq:23}
	x f_{n-1}(x) = \sum_{j=0}^{n} J_{n-1, j} f_j(x).
\end{equation}
Observe that by \eqref{eq:23} and \eqref{eq:2a} $f_n(x)$ and $p_n(x)$ satisfy the same recurrence relation. Since $f_0(x) = p_0(x)$ we obtain that $f_n(x) = p_n(x)$.
\end{proof}

Analogously, one can try to consider the matrix representation of $M_x$ on the space $\calQ$, where
\[
	\calQ_n := \lin \{ q_{k} : k = 0, 1, \ldots, n \} \quad \text{and} \quad
	\calQ := \bigcup_{k=0}^\infty \calQ_k.
\]
The following proposition provides sufficient conditions under which it is possible.
\begin{proposition} \label{prop:1}
Let $(\vec{n}_\ell : \ell \geq 0)$ be a path satisfying \eqref{eq:24}. Let $(p_\ell : \ell \geq 0)$ and $(q_\ell : \ell \geq 0)$ be defined as in \eqref{eq:25}. If
\begin{equation} \label{eq:11}
	\lim_{\ell \to \infty} (\vec{n}_\ell)_j = \infty, \quad j=1, 2, \ldots r,
\end{equation}
then the operator $M_x$ is well-defined on $\calQ$. Moreover, for any $\ell \geq 0$ we have
\[
	M_x q_\ell = \sum_{k=0}^\infty J_{k, \ell} q_k,
\]
where $J$ is defined in \eqref{eq:2a}, and the above sum contains a finite number of non-zero elements.
\end{proposition}
\begin{proof}
First of all, by \cite[Corollary 23.1.1]{Ismail2009} and induction we can derive that for any $\ell \geq 0$
\begin{equation} \label{eq:12}
	\lin \big\{ A_{\vec{n}_{k+1}} : 0 \leq k \leq \ell \big\} = 
	\bigoplus_{j=1}^r \calP_{(\vec{n}_N)_j}.
\end{equation}
In particular,
\[
	A_{\vec{n}_{\ell+1}} \in \bigoplus_{j=1}^r \calP_{(\vec{n}_\ell)_j},
\]
and consequently,
\begin{equation} \label{eq:13}
	x A_{\vec{n}_{\ell+1}} \in \bigoplus_{j=1}^r x \calP_{(\vec{n}_\ell)_j}
	\subseteq
	\bigoplus_{j=1}^r \calP_{(\vec{n}_\ell)_j+1}.
\end{equation}
Now, by \eqref{eq:11} there exists \emph{minimal} $N_\ell \geq \ell+r$ such that
\[
	(\vec{n}_{N_\ell})_j \geq (\vec{n}_\ell)_j + 1, \quad j=1,2,\ldots,r.
\]
Thus, by \eqref{eq:12} and \eqref{eq:13} we get
\[
	x A_{\vec{n}_\ell+1} \in \lin \big\{ A_{\vec{n}_{k+1}} : 0 \leq k \leq N_\ell \big\},
\]
which immediately implies $x q_\ell \in \calQ_{N_\ell}$. Thus, there are constants $\{ S_{\ell,k} \}_{k=0}^{N_\ell}$ such that
\begin{equation} \label{eq:2b}
	M_x q_\ell = \sum_{k=0}^{N_\ell} S_{\ell,k} q_k = 
	\sum_{k=0}^\infty S_{\ell,k} q_k,
\end{equation}
where we have defined $S_{\ell,k} = 0$ for $k > N_\ell$. Let us collect these constants into a matrix $S=[S_{\ell,k}]_{\ell,k=0}^\infty$. It remains to prove that $S = J^t$.
Let us take the scalar product on the both sides of \eqref{eq:2a} with $q_{\ell'}$. Then
\[
	\langle M_x p_{\ell}, q_{\ell'} \rangle_{L^2(\mu)} = 
	\sum_{k=0}^{\infty} J_{\ell, k} \langle p_{k}, q_{\ell'} \rangle_{L^2(\mu)} =
	J_{\ell, \ell'},
\]
where the last equality follows from \eqref{eq:1}. Next, by self-adjointness of $M_x$, we have
\[
	\langle M_x p_{\ell}, q_{\ell'} \rangle_{L^2(\mu)} =
	\langle p_{\ell}, M_x q_{\ell'} \rangle_{L^2(\mu)} =
	{\langle M_x q_{\ell'}, p_{\ell} \rangle}_{L^2(\mu)}.
\]
Hence, by an analogous reasoning applied to \eqref{eq:2b}, we obtain
\[
	\langle M_x q_{\ell'}, p_\ell \rangle_{L^2(\mu)} =
	\sum_{k=0}^\infty S_{\ell',k} \langle q_k, p_{\ell} \rangle_{L^2(\mu)} =
	S_{\ell', \ell}
\]
Thus, $S = {J^t}$ and the proof is complete.
\end{proof}

\section{Near diagonal boundedness of $J$} \label{sec:5}
Let us denote by $\Hes$ the set of lower Hessenberg matrices. Specifically, 
\[
	X \in \Hes \quad \Leftrightarrow \quad X_{i,j} = 0 \ \text{for any } j > i+1.
\]
In particular, we have $J \in \Hes$. Let us define the neighbourhood of the diagonal in $\NN_0^2$ of radius 
$R \geq 0$ by
\[
	D_R = \big\{ (i,j) \in \NN_0^2 : | i-j | \leq R \big\}.
\]
In what follows, we will need conditions implying that
\begin{equation} \label{eq:NDB}
	\tag{NDB}
	\sup_{(i,j) \in D_R} |X_{i,j}| < \infty \quad \text{for any } R \geq 0.
\end{equation}
The following proposition implies that if $X \in \Hes$ satisfies \eqref{eq:NDB}, then every power of $X$ has this property.
\begin{proposition} \label{prop:3}
Let $X \in \Hes$. Then for any $\ell \geq 1$ and $M \geq 0$ there are constants $c(M,\ell) \geq 1$ and $R(M,\ell) \geq M$ \emph{independent} of $X$ such that
\begin{equation} \label{eq:8}
	\sup_{(i,j) \in D_M} | [X^\ell]_{i,j} | \leq 
	c(M, \ell) \Big( \sup_{(i,j) \in D_{R(M, \ell)}} |X_{i,j}| \Big)^\ell.
\end{equation}
\end{proposition}
\begin{proof}
First observe that by induction one can show that $[X^\ell]_{i,j} = 0$ provided $j > i+\ell$.
Thus for any $\ell \geq 2$
\begin{equation} \label{eq:9}
	[X^\ell]_{i,j} = 
	\sum_{k=0}^\infty [X^{\ell-1}]_{i,k} X_{k,j} =
	\sum_{k=j-1}^{i+\ell-1} [X^{\ell-1}]_{i,k} X_{k,j}.
\end{equation}

We shall prove \eqref{eq:8} inductively. For $\ell=1$, the statement holds true for $c(M, \ell) = 1$
and $R(M, \ell) = M$.
Suppose that $\ell \geq 2$. Then by \eqref{eq:9} and for any $(i,j) \in D_M$ we have
\[
	| [X^\ell]_{i,j} | \leq (M+\ell+1) \cdot
	\sup_{(i,k) \in D_{M+\ell}} | [X^{\ell-1}]_{i,k} |  \cdot
	\sup_{(k,j) \in D_{M+\ell}} |X_{k,j}|.
\]
By the induction hypothesis
\[
	| [X^\ell]_{i,j} | \leq (M+\ell+1) \cdot c(M+\ell, \ell-1)
	\Big( \sup_{(i,k) \in D_{R(M+\ell, \ell-1)}} | X_{i,k} | \Big)^{\ell-1} \cdot
	\sup_{(k,j) \in D_{M+\ell}} |X_{k,j}|.
\]
Hence by defining
\[
	c(M, \ell) := (M+\ell+1) \cdot c(M+\ell, \ell-1) \quad \text{and} \quad
	R(M, \ell) := \max \big( R(M+\ell, \ell-1), M+\ell \big)
\]
we obtain
\begin{equation} \label{eq:10}
	| [X^\ell]_{i,j} | \leq c(M, \ell)
	\Big( \sup_{(i,k) \in D_{R(M, \ell)}} | X_{i,k} | \Big)^{\ell} .
\end{equation}
The conclusion follows by taking the supremum over all $(i,j) \in D_M$
on the left-hand side of \eqref{eq:10}. 
\end{proof}

\subsection{Criteria for the boundedness}
The first sufficient condition implying that $J$ satisfies \eqref{eq:NDB} is formulated in terms of the nearest neighbor recurrence relations which multiple orthogonal polynomials satisfy. Let us recall that according to \cite{VanAssche2011} we have for the type II polynomials
\begin{equation}   \label{NNRR-P}
    xP_{\vec{n}}(x) = 
    P_{\vec{n}+\vec{e}_k}(x) + 
    b_{\vec{n},k} P_{\vec{n}}(x) + 
    \sum_{j=1}^r a_{\vec{n},j} P_{\vec{n}-\vec{e}_j}(x), \qquad 1 \leq k \leq r
\end{equation}
for some real sequences $a_{\vec{n},j}$ and $b_{\vec{n},k}$.

More precisely, we have the following
\begin{proposition}  \label{prop:4}
Let $(\vec{n}_{\ell} : \ell \geq 0)$ be a path satisfying \eqref{eq:24}.
Suppose that the nearest neighbour recurrence coefficients for MOPs satisfy
\begin{align}
	\label{eq:3a}
	&\max_{1 \leq i \leq r} \sup_{\ell \geq 0} |a_{\vec{n}_\ell, i}| < \infty, \\
	\label{eq:3b}
	&\max_{1 \leq i \leq r} \sup_{\vec{n} \in \NN_0^r} |b_{\vec{n},i}| < \infty.
\end{align}
Then the matrix $J$ associated with the sequence $(\vec{n}_{\ell} : \ell \geq 0)$ satisfies \eqref{eq:NDB}.
\end{proposition}
\begin{proof}
In view of the formula \eqref{eq:2a} and \cite[Proposition 2.1]{Duits2021}, the matrix $J$ satisfies
\[
	\begin{gathered}
	J_{\ell,\ell+1} = 1, \qquad  
	J_{\ell, \ell} = b_{\vec{n}_\ell, i_{\ell}} \\
	J_{\ell, j} =
	\sum_{k=1}^r a_{\vec{n}_{\ell},k} 
	\prod_{m=j+2}^\ell
	\big( b_{\vec{n}_{m-1}-\vec{e}_k,k} - b_{\vec{n}_{m-1}-\vec{e}_k,i_{m-1}} \big), \quad 0 \leq j \leq \ell-1.	
	\end{gathered}
\]
Let $R \geq 0$. Then for $(\ell, j) \in D_R$ we have $|j-\ell| \leq R$. Hence the product above has at most $R-1$ terms which, by \eqref{eq:3b}, are uniformly bounded. Hence, the result readily follows from \eqref{eq:3a}.
\end{proof}

The following theorem gives other conditions implying \eqref{eq:NDB}. Conditions implying the hypotheses of this result are contained in \cite{Haneczok2012}.
\begin{theorem} \label{thm:3}
Let $(\vec{n}_{\ell} : \ell \geq 0)$ be a path satisfying \eqref{eq:24} and let $J$ be the corresponding matrix 
defined by \eqref{eq:2a}.
Suppose there is a compact interval $\Delta \subset \RR$ such that for any $n \in \NN_0$:
\begin{enumerate}[(a)]
	\item the polynomial $p_n$ has exactly $n$ simple zeros which lie inside $\Delta$,
	\item the zeros of $p_n$ and $p_{n+1}$ are interlacing.
\end{enumerate}
Then the matrix $J$ satisfies \eqref{eq:NDB}.
\end{theorem}
\begin{proof}
It is enough to prove that 
\begin{equation} \label{eq:6}
	\sup_{n \geq N} |J_{n,n-N}| < \infty \quad \text{for any } N \geq 0.
\end{equation}
We will prove this inductively with respect to $N$.
We follow the argument from \cite[Lemma 2.2]{Aptekarev2006} and \cite[Lemma 2.3.2]{Denisov2022}.

Let the zeros of $p_n$ be denoted by $(x_{k;n} : k =1,\ldots, n)$, where we order them as follows
\[
	x_{1;n} < x_{2;n} < \ldots < x_{n;n}.
\]
Let $p_{n+1}(z) = (z - x_{1;n+1}) (z - x_{n+1;n+1}) r_n(z)$. Then by a partial fraction decomposition we obtain
\begin{align*}
	\frac{p_{n+1}(z)}{p_n(z)} 
	&= 
	(z - x_{1;n+1}) (z - x_{n+1;n+1}) \frac{r_n(z)}{p_{n}(z)} \\
	&=
	(z - x_{1;n+1}) (z - x_{n+1;n+1}) 
	\sum_{k=1}^n \frac{r_n(x_{k;n})}{p_{n}'(x_{k;n})} \frac{1}{z-x_{k;n}} .
\end{align*}
Let us divide the last identity by $z$. Since the polynomials $p_n$ are monic we get
\begin{align*}
	1 &= 
	\lim_{|z| \to \infty} \frac{p_{n+1}(z)}{z p_n(z)} \\
	&= 
	\lim_{|z| \to \infty} \frac{z - x_{1;n+1}}{z}
	\sum_{k=1}^n \frac{r_n(x_{k;n})}{p_{n}'(x_{k;n})} 
	\lim_{|z| \to \infty} \frac{z - x_{n+1;n+1}}{z-x_{k;n}} \\
	&=
	\sum_{k=1}^n \frac{r_n(x_{k;n})}{p_{n}'(x_{k;n})}.
\end{align*}
Since the polynomials $r_n$ and $p_n$ are monic and their zeros are interlacing one can prove that
\[
	\frac{r_n(x_{k;n})}{p_{n}'(x_{k;n})} > 0.
\]
Hence for any compact $K \subset \CC \setminus \Delta$
\[
	M_1(K) := \sup_{n \geq 0} \sup_{z \in K} \frac{|p_{n+1}(z)|}{|p_n(z)|} < \infty.
\]
Consequently, for any $j \in \NN$
\begin{equation} \label{eq:14}
	M_j(K) := \sup_{n \geq 0} \sup_{z \in K} \frac{|p_{n+j}(z)|}{|p_n(z)|} \leq M_1(K)^j < \infty.
\end{equation}

Let us turn to the proof of \eqref{eq:6} for $N=0$. Since $(p_n : n \geq 0)$ are monic, by \eqref{eq:2a} we have
\begin{equation} \label{eq:19}
	x p_n = p_{n+1} + \sum_{k=0}^n J_{n,k} p_k.
\end{equation}
Hence, by dividing both sides by $x p_n$ we get
\begin{equation} \label{eq:15}
	1 - \frac{p_{n+1}}{x p_n} = \sum_{k=0}^n J_{n,k} \frac{p_k}{x p_n}.
\end{equation}
Let us denote 
\begin{equation} \label{eq:16}
	f(z) = 1 - \frac{p_{n+1}(z)}{z p_n(z)}.
\end{equation}
Since $(p_n : n \geq 0)$ are monic we get by \eqref{eq:15} and \eqref{eq:16}
\begin{equation} \label{eq:17}
	\lim_{|z| \to \infty} f(z) = 0 \quad \text{and} \quad
	\lim_{|z| \to \infty} z f(z) = J_{n,n}.
\end{equation}
Thus, we arrive at 
\begin{equation} \label{eq:18}
	\Res(f, \infty) = -J_{n,n}.
\end{equation}
On the other hand, let $\gamma_r$ be a positively oriented circle around the origin with radius $r$ which contains $\Delta$ in its interior. Then by the residue theorem
\begin{align*}
	\Res(f, \infty) 
	&= 
	-\frac{1}{2 \pi i} \int_{\gamma_r} f(z) \ud z \\
	&= 
	-\frac{1}{2 \pi i} \int_{\gamma_r} \bigg( 1 - \frac{p_{n+1}(z)}{z p_n(z)} \bigg) \ud z \\
	&=
	\frac{1}{2 \pi i} \int_{\gamma_r} \frac{p_{n+1}(z)}{p_n(z)} \frac{\ud z}{z}.
\end{align*}
Hence, by \eqref{eq:14} and \eqref{eq:18} we obtain
\[
	|J_{n,n}| \leq M_1(\gamma_r) =: C_0.
\]

Now, let $N \geq 1$ and suppose that
\begin{equation} \label{eq:22}
	\sup_{n \geq 0} |J_{n,n-j}| \leq C_j, \quad j=0, 1, \ldots N-1.
\end{equation}
We shall prove a similar bound for $|J_{n, n-N}|$. Similarly as before, let us divide both sides of \eqref{eq:19} by $x p_{n-N}$. Then
\[
	\frac{p_{n}}{p_{n-N}} - 
	\frac{p_{n+1}}{x p_{n-N}} - 
	\sum_{k=n-N+1}^{n} J_{n,k} \frac{p_k}{x p_{n-N}} = 
	\sum_{k=0}^{n-N} J_{n,k} \frac{p_k}{x p_{n-N}}.
\]
Let us denote
\[
	f(z) = 
	\frac{p_{n}(z)}{p_{n-N}(z)} - 
	\frac{p_{n+1}(z)}{z p_{n-N}(z)} - 
	\sum_{k=n-N+1}^{n} J_{n,k} \frac{p_k(z)}{z p_{n-N}(z)}.
\]
Then
\[
	\lim_{|z| \to \infty} f(z) = 0 \quad \text{and} \quad
	\lim_{|z| \to \infty} z f(z) = J_{n, n-N}.
\]
Hence
\begin{equation} \label{eq:20}
	\Res(f, \infty) = -J_{n, n-N}.
\end{equation}
On the other hand, analogously as before,
\begin{equation} \label{eq:21}
	\Res(f, \infty) =
	-\frac{1}{2 \pi i}
	\int_{\gamma} z f(z) \frac{\ud z}{z}.
\end{equation}
But by \eqref{eq:14} and \eqref{eq:22}
\[
	\sup_{z \in \gamma} |z f(z)| \leq 
	r M_N(\gamma_r) + M_{N+1}(\gamma_r) + \sum_{k=0}^{N-1} C_{k} M_{k+1}(\gamma_r) =: C_N.
\]
Hence by \eqref{eq:20} and \eqref{eq:21} we get
\[
	\sup_{n \geq N} |J_{n, n-N}| \leq C_N
\]
which ends the proof of \eqref{eq:6} for $N$.    
\end{proof}

\section{Density of zeros and weak limit of the Christoffel--Darboux kernel} \label{sec:6}
Let us recall that in \eqref{eq:49} we have defined
\begin{equation} \label{eq:27}
	\nu_n = \frac{1}{n} \sum_{y \in p_n^{-1}[\{0\}]} \delta_y \quad \text{and} \quad
	\ud \eta_n(x) = \frac{1}{n} K_n(x,x) \ud \mu(x),
\end{equation}
where in the first sum we take into account the possible multiplicities of the zeros of $p_n$, and
\begin{equation} \label{eq:28}
	K_n(x,y) = \sum_{j=0}^{n-1} p_j(x) {q_j(y)}.
\end{equation}

The proof of the following theorem is based on \cite[Proposition 2.3]{Simon2009}.
\begin{theorem} \label{thm:4}
Let $(\vec{n}_\ell : \ell \geq 0)$ be a path satisfying \eqref{eq:24}. Suppose that the corresponding matrix $J$, 
defined by \eqref{eq:2a}, satisfies \eqref{eq:NDB}. Then for any $\ell \geq 0$
\[
	\lim_{n \to \infty}
	\bigg| 
	\int_\RR x^\ell \ud \nu_{n} - \int_\RR x^\ell \ud \eta_n \bigg| = 0.
\]
\end{theorem}
\begin{proof}
By Proposition~\ref{prop:2} we obtain
\[
	\int_\RR x^\ell \ud \nu_{n} = 
	\frac{1}{n} \tr (J_n^\ell) =
	\frac{1}{n} \sum_{j=0}^{n-1} [ J_n^\ell ]_{j,j} .
\]
By \eqref{eq:2a} we have
\begin{align*}
	\int_\RR x^\ell \ud \eta_n &= 
	\frac{1}{n} \sum_{j=0}^{n-1} \int_\RR x^\ell p_j(x) {q_j(x)} \ud \mu(x) \\
	&= 
	\frac{1}{n} \sum_{j=0}^{n-1} \big\langle M_x^\ell p_j, q_j \big\rangle_{L^2(\mu)} \\
	&=
	\frac{1}{n} \sum_{j=0}^{n-1} [ J^\ell ]_{j,j} .
\end{align*}
Hence
\begin{align}
	\nonumber
	\bigg| 
	\int_\RR x^\ell \ud \nu_{n} - \int_\RR x^\ell \ud \eta_n \bigg| &\leq
	\frac{1}{n} 
	\sum_{j=0}^{n-1} 
	\big| [ J_n^\ell ]_{j,j} - [ J^\ell ]_{j,j} \big| \\
	\nonumber
	&=
	\frac{1}{n} 
	\sum_{j=n-\ell}^{n-1} 
	\big| [ J_n^\ell ]_{j,j} - [ J^\ell ]_{j,j} \big| \\
	\label{eq:26}
	&\leq
	\frac{1}{n}
	\sum_{j=n-\ell}^{n-1} 
	\big( |[ J_n^\ell ]_{j,j}| + |[ J^\ell ]_{j,j}| \big).
\end{align}
By \eqref{eq:NDB} and Proposition~\ref{prop:3} we obtain that any fixed $\ell$ 
\[
	\sup_{(i,j) \in D_M} |[J^\ell]_{i,j}| < \infty \quad \text{for any } M \geq 0.
\]
Next, since
\[
	|[J_n]_{i,j}| \leq |J_{i,j}|,
\]
Proposition~\ref{prop:3} together with \eqref{eq:NDB} implies the existence of constants $c(M,\ell) \geq 1$ and $R(M, \ell) \geq M$
such that
\begin{align*}
	\sup_{(i,j) \in D_M} |[J_n^\ell]_{i,j}| 
	&\leq 
	c(M, \ell) \Big( \sup_{(i,j) \in D_{R(M, \ell)}} |[J_n]_{i,j}| \Big)^\ell \\
	&\leq
	c(M, \ell) \Big( \sup_{(i,j) \in D_{R(M, \ell)}} |J_{i,j}| \Big)^\ell .
\end{align*}
Hence
\[
	\sup_{n \geq 1} \sup_{(i,j) \in D_M} |[J_n^\ell]_{i,j}| < \infty \quad \text{for any } M \geq 0.
\]
Thus,
\[
	\sup_{n \geq 1} 
	\sup_{n-\ell \leq j < n} 
	\big( [ |J_n^\ell ]_{j,j}| + |[ J^\ell ]_{j,j}| \big) < \infty.
\]
Hence, together with \eqref{eq:26} this implies the existence of a constant $c > 0$ such that
\[
	\bigg| 
	\int_\RR x^\ell \ud \nu_{n} - \int_\RR x^\ell \ud \eta_n \bigg|
	\leq \frac{c}{n}, 
\]
from which the conclusion follows.
\end{proof}

\begin{corollary} \label{cor:1}
Let the hypotheses of Theorem~\ref{thm:4} be satisfied. Let $(n_k : k \in \NN_0)$ be an increasing sequence of positive integers. Suppose that there is a compact set $K \subset \RR$ such that $\supp(\mu) \subset K$ and $\supp(\nu_{n_k}) \subset K$ for any $k \in \NN_0$. If
\begin{equation} \label{eq:4}
	\sup_{k \in \NN_0} \frac{1}{n_k} \int_\RR |K_{n_k}(x,x)| \ud \mu(x) < \infty,
\end{equation}
then 
\begin{equation} \label{eq:5'}
	\lim_{k \to \infty} \bigg| \int_\RR f \ud \nu_{n_k} - \int_\RR f \ud \eta_{n_k} \bigg| = 0, \quad f \in \calC(K).
\end{equation}
In particular, for any probability measure $\nu_\infty$ supported on $K$ one has the equivalence\footnote{By $\mu_{n} \xrightarrow{w} \mu$ we denote the weak convergence of finite measures, i.e. $\int_\RR f \ud \mu_n \to \int_\RR f \ud \mu$ for any $f \in \calC_b(\RR)$.}
\begin{equation} \label{eq:5}
	\nu_{n_k} \xrightarrow{w} \nu_\infty \quad \Leftrightarrow \quad
	\eta_{n_k} \xrightarrow{w} \nu_\infty.
\end{equation}
\end{corollary}
\begin{proof}
First of all, it is immediate from the definition \eqref{eq:27} that $(\nu_{n_k} : k \geq 1)$ is a sequence of probability measures. Let us observe that the sequence $(\eta_{n_k} : k \geq 1)$ is bounded in the total variation norm. 
Indeed, in view of \eqref{eq:27}, we have
\[
	\| \eta_{n_k} \|_{\mathrm{TV}} =
	|\eta_{n_k}|(\RR) = \frac{1}{n_k} \int_\RR |K_{n_k}(x,x)| \ud \mu(x).
\]
Hence, by \eqref{eq:4}, we obtain
\begin{equation} \label{eq:4'}
	\sup_{k \geq 0} \| \eta_{n_k} \|_{\mathrm{TV}} =: C < \infty.
\end{equation}

Take $f \in \calC(K)$. Then, by the Weierstrass theorem, for any $\epsilon > 0$, 
there exists $P_\epsilon \in \RR[x]$ such that
\begin{equation} \label{eq:7a}
	\sup_{x \in K} |f(x) - P_\epsilon(x)| < \epsilon.
\end{equation}
Hence, by \eqref{eq:7a} and \eqref{eq:4'},
\begin{align*}
	\bigg|\int_\RR f \ud \nu_{n_k} - \int_\RR f \ud \eta_{n_k} \bigg| 
	&=
	\bigg| \int_\RR (f - P_\epsilon) \ud \nu_{n_k} - \int_\RR (f-P_\epsilon) \ud \eta_{n_k} + \int_\RR P_\epsilon \ud \nu_{n_k} - \int_\RR P_\epsilon \ud \eta_{n_k} \bigg| \\
	&\leq	
	( 1 + C) \epsilon + 
	\bigg| \int_\RR P_\epsilon \ud \nu_{n_k} - \int_\RR P_\epsilon \ud \eta_{n_k} \bigg|.
\end{align*}
Thus, by Theorem~\ref{thm:4} we get
\[
	\lim_{k \to \infty}
	\bigg|\int_\RR f \ud \nu_{n_k} - \int_\RR f \ud \eta_{n_k} \bigg| \leq
	(1+C) \epsilon .
\]
By letting $\epsilon \to 0$ we obtain \eqref{eq:5'}. The conclusion \eqref{eq:5} follows from this immediately.
\end{proof}

Let us comment that the condition \eqref{eq:4} seems not to be automatically true. 
Observe however that in view of \eqref{eq:27}, \eqref{eq:28} and \eqref{eq:1} we have
\[
	\frac{1}{n_k} \int_\RR K_{n_k}(x,x) \ud \mu(x) = \frac{n_k}{n_k} = 1.
\]
Hence, if
\begin{equation} \label{eq:POS}
	\tag{POS}
	K_{n_k}(x,x) \geq 0 \quad \text{for a.e. } x \in \supp(\mu),\ k \geq 0,
\end{equation}
then \eqref{eq:4} is satisfied.

A direct consequence of \cite[Section 3.2 and 4.1]{Kuijlaars2010} is the following result.
\begin{lemma} \label{lem:2}
Let $n \in \NN$ be given and suppose that for any points $\{x_1, \ldots, x_n \} \subset \supp(\mu)$ satisfying $x_1 < \ldots < x_n$ one has
\begin{equation} \label{eq:detPOS}
	\det \big[ K_n(x_i, x_j) \big]_{i,j=1,\ldots,n} \geq 0.
\end{equation}
Then $K_n(x,x) \geq 0$ for any $x \in \supp(\mu)$, and consequently, \eqref{eq:4} is satisfied.
\end{lemma}

According to Kuijlaars \cite{Kuijlaars2010}, the condition \eqref{eq:detPOS} implies the existence of a determinantal point process whose correlation kernel is equal to $K_n(\cdot,\cdot)$. Let us mention that \cite[Section 4]{Kuijlaars2010} is devoted to discussing conditions under which \eqref{eq:detPOS} is satisfied. In particular, this condition is always satisfied for $r=1$ where it leads to the so-called \emph{orthogonal polynomial ensembles}. A survey of some models in probability theory leading to orthogonal polynomial ensembles has been given in \cite{Konig2005}.

\section{The Nevai condition} \label{sec:7}
Let $(\vec{n}_\ell : \ell \geq 0)$ be a path satisfying \eqref{eq:24}. Let $K_n$ be the corresponding Christoffel--Darboux kernel.
Inspired by Nevai \cite[Section 6.2]{Nevai1979} (see also \cite{Breuer2010a}) we define for every bounded measurable function $f$ and $x \in \supp(\mu)$
\[
	G_n[f](x) = \frac{1}{K_n(x,x)} \int_\RR K_n(x,y) K_n(y,x) f(y) \ud \mu(y) ,
\]
provided that $K_n(x,x) \neq 0$. 

\begin{proposition}
Let $n \in \NN$ and $x \in \supp(\mu)$. Suppose that $K_n(x,x) \neq 0$. Then the definition of $G_n[f]$ makes sense for any bounded measurable function $f$.
\end{proposition}
\begin{proof}
We have 
\[
	\big| G_n[f](x) \big| \leq 
	\frac{1}{|K_n(x,x)|} \sup_{y \in \RR} |f(y)|
	\int_\RR |K_n(x,y) K_n(y,x)| \ud \mu(y).
\]
Hence, we need to prove
\begin{equation} \label{eq:45}
	\int_\RR |K_n(x,y) K_n(y,x)| \ud \mu(y) < \infty.
\end{equation}
To do so, by the Cauchy--Schwarz inequality we obtain
\[
	\int_\RR |K_n(x,y) K_n(y,x)| \ud \mu(y) \leq
	\bigg( \int_\RR |K_n(x,y)|^2 \ud \mu(y) \bigg)^{1/2}
	\bigg( \int_\RR |K_n(y,x)|^2 \ud \mu(y) \bigg)^{1/2}.
\]
Next, by the Cauchy--Schwarz inequality applied to \eqref{eq:28} gives
\[
	|K_n(x,y)|^2 \leq 
	\bigg( \sum_{j=0}^{n-1} |p_j(x)|^2 \bigg) \cdot
	\bigg( \sum_{j=0}^{n-1} |q_j(y)|^2 \bigg).
\]
Hence
\[
	\int_\RR |K_n(x,y)|^2 \ud \mu(y) \leq
	\bigg( \sum_{j=0}^{n-1} |p_j(x)|^2 \bigg) \cdot
	\bigg( \sum_{j=0}^{n-1} \|q_j\|_{L^2(\mu)}^2 \bigg) < \infty.
\]
Similarly we obtain
\[
	\int_\RR |K_n(y,x)|^2 \ud \mu(y) \leq
	\bigg( \sum_{j=0}^{n-1} |q_j(x)|^2 \bigg) \cdot
	\bigg( \sum_{j=0}^{n-1} \|p_j\|_{L^2(\mu)}^2 \bigg) < \infty,
\]
which implies \eqref{eq:45}.
\end{proof}

We say that $(K_n : n \geq 0)$ satisfies the \emph{Nevai condition} at $x \in \supp(\mu)$
if
\begin{equation} \label{eq:38}
	\lim_{n \to \infty}
	G_n[f](x) = 
	f(x), \quad f \in \calC_b(\RR).
\end{equation}
Our motivation of examining the condition \eqref{eq:38} comes from the desire to understand the pointwise behaviour of $K_n(x,x)$. A very modest result in this direction is the following:
\begin{proposition} \label{prop:6}
Let $x \in \supp(\mu)$ be an isolated point. If \eqref{eq:38} is satisfied for $x$, then
\begin{equation} \label{eq:43}
	\lim_{n \to \infty} K_n(x,x) = \frac{1}{\mu(\{x\})}.
\end{equation}
\end{proposition}
\begin{proof}
Let $\varepsilon > 0$ be such that $[x-\varepsilon, x+\varepsilon] \cap \supp(\mu) = \{x\}$.
Set $f(y) = \frac{1}{\varepsilon} \max(\varepsilon - |y-x|, 0)$. Then $f \in \calC_b(\RR)$, $\supp(f) = [x-\varepsilon,x+\varepsilon]$ and $f(x) = 1$. Hence
\[
	\frac{1}{K_n(x,x)} \int_\RR K_n(x,y) K_n(y,x) f(y) \ud \mu(y) = 
	K_n(x,x) \mu(\{x\}).
\]
On the other hand, by \eqref{eq:38}, we get
\[
	\lim_{n \to \infty} \frac{1}{K_n(x,x)} \int_\RR K_n(x,y) K_n(y,x) f(y) \ud \mu(y) = 1.
\]
By combining the last two formulas the result follows.
\end{proof}
In the classical case $r=1$, the formula \eqref{eq:43} holds for any $x \in \supp(\mu)$ (with the convention that $1/0 = +\infty$) provided the measure $\mu$ is determined by its moments, see e.g. \cite[Corollary 2.6]{Shohat1943}. This motivates the following problem.

\begin{problem} \label{prob:1}
Suppose that the support of $\mu$ is compact. Is it true that \eqref{eq:43} holds for any $x \in \supp(\mu)$?
\end{problem}

Let us observe that \eqref{eq:43} implies the positivity of $K_n(x,x)$ for large $n$. Hence, the positive answer to Problem~\ref{prob:1} leads to a weaker version of \eqref{eq:POS}.

Our aim is to prove that under some hypotheses \eqref{eq:38} holds for any $f \in \RR[x]$. It turns out that in this situation $G_n[f]$ can be rewritten in terms of $J$.
\begin{lemma} \label{lem:1} 
For any $k \geq 0$ we have
\[
	\int_\RR K_n(x,y) y^k K_n(y,x) \ud \mu(y) =
	\big\langle 
		\Pi_n \vec{q}(x), 
		J^k \Pi_n \vec{p}(x) 
	\big\rangle_{\ell^2},
\]
where the sequence $\Pi_n u$ is defined by
\[
	[\Pi_n u]_k =
	\begin{cases}
		u_k, & k < n, \\
		0, & \text{otherwise},
	\end{cases}
\]
and $\vec{p}(x)$ and $\vec{q}(x)$ are defined as follows
\[
	\vec{p}(x) = (p_0(x),p_1(x), \ldots)^t, \quad
	\vec{q}(x) = (q_0(x), q_1(x), \ldots)^t.
\]
\end{lemma}
\begin{proof}
We have
\begin{align*}
	\int_\RR K_n(x,y) y^k K_n(y,x) \ud \mu(y) 
	&=
	\big\langle K_n(x,y), y^k K_n(y,x) \big\rangle_{L^2_y(\mu)} \\
	&=
	\bigg\langle 
		\sum_{\ell=0}^{n-1} p_\ell(x) q_{\ell}(y), 
		y^k \sum_{\ell'=0}^{n-1} p_{\ell'}(y) q_{\ell'}(x) 
	\bigg\rangle_{L^2_{y}(\mu)} \\
	&=
	\sum_{\ell=0}^{n-1} \sum_{\ell'=0}^{n-1} 
	p_{\ell}(x) q_{\ell'}(x) 
	\big\langle 
		q_{\ell}, 
		y^k p_{\ell'} 
	\big\rangle_{L^2(\mu)}.
\end{align*}
Let us recall that $\langle y p_{\ell'}, q_{\ell} \rangle_{L^2(\mu)} = J_{\ell', \ell}$. Hence we arrive at
\begin{align*}
	\int_\RR K_n(x,y) y^k K_n(y,x) \ud \mu(y) 
	&=
	\sum_{\ell=0}^{n-1} \sum_{\ell'=0}^{n-1} 
	p_{\ell}(x) q_{\ell'}(x) [J^k]_{\ell', \ell} \\
	&=
	\big\langle
		\Pi_n \vec{q}(x),
		J^k \Pi_n \vec{p}(x)
	\big\rangle_{\ell^2}. \qedhere
\end{align*}
\end{proof}

We are ready to prove our main result of this section.
\begin{theorem} \label{thm:5}
Assume that:
\begin{enumerate}[(a)]
\item the matrix $J$ satisfies \eqref{eq:NDB},
\item $K_n(x,x) > 0$ for large $n$,
\item for any $N \geq 1$ one has
$\begin{aligned}[b]
	\lim_{n \to \infty} \frac{q_{\ell}(x) p_{\ell'}(x)}{K_n(x,x)} = 0
\end{aligned}$, where $\ell' \in [n-N, n)$ and $\ell \in [n, n+N]$.
\end{enumerate}
Then the convergence \eqref{eq:38} holds for any polynomial $f \in \RR[x]$.
\end{theorem}
\begin{proof}
By linearity it is enough to prove the convergence for $f(x) = x^k$ for any $k \geq 0$.
To do so, let us recall that $J \vec{p}(x) = x \vec{p}(x)$. Thus
\[
	x^k K_n(x,x) 
	=
	\big\langle
		\Pi_n \vec{q}(x),
		\Pi_n x^k \vec{p}(x)
	\big\rangle_{\ell^2}
	=
	\big\langle
		\Pi_n \vec{q}(x),
		\Pi_n J^k \vec{p}(x)
	\big\rangle_{\ell^2}.
\]
Hence, in view of Lemma~\ref{lem:1}, we need to estimate
\[
	\int_\RR K_n(x,y) y^k K_n(y,x) \ud \mu(y) - x^k K_n(x,x) =
	\big\langle
		\Pi_n \vec{q}(x),
		(J^k \Pi_n - \Pi_n J^k) \vec{p}(x)
	\big\rangle_{\ell^2}.
\]
Let us observe that since $J \in \Hes$ we have $[J^k]_{\ell, \ell'} = 0$ for $\ell' > \ell+k$.
Thus,
\begin{align*}
	\big\langle
		\Pi_n \vec{q}(x),
		J^k \Pi_n \vec{p}(x)
	\big\rangle_{\ell^2}
	&=
	\sum_{\ell=0}^{n-1} q_{\ell}(x) \sum_{\ell'=0}^{n-1} [J^k]_{\ell, \ell'} p_{\ell'}(x) \\
	&= 
	\sum_{\ell=0}^{n-1} q_{\ell}(x) \sum_{\ell'=0}^{\min(n-1,\ell+k)} [J^k]_{\ell, \ell'} p_{\ell'}(x),
\end{align*}
and
\[
	\big\langle
		\Pi_n \vec{q}(x),
		\Pi_n J^k \vec{p}(x)
	\big\rangle_{\ell^2} 
	=
	\sum_{\ell=0}^{n-1} q_\ell(x) \sum_{\ell'=0}^{\ell+k} [J^k]_{\ell, \ell'} p_{\ell'}(x).
\]
Hence
\[
	\big\langle
		\Pi_n \vec{q}(x),
		(J^k \Pi_n - \Pi_n J^k) \vec{p}(x)
	\big\rangle_{\ell^2}
	=
	-\sum_{\ell=n-k}^{n-1} q_{\ell}(x) \sum_{\ell'=n}^{\ell+k} [J^k]_{\ell, \ell'} p_{\ell'}(x).
\]
Now, in view of Proposition~\ref{prop:3}, the result easily follows.
\end{proof}

If in the setup of Theorem~\ref{thm:5} we additionally assume that $\supp(\mu)$ is compact and
\begin{equation} \label{eq:39}
	\sup_{n \geq 0} 
	\frac{1}{K_n(x,x)} 
	\int_{\RR} |K_n(x,y) K_n(y,x)| \ud \mu(y) < \infty,
\end{equation}
then by a similar reasoning as in the proof of Corollary~\ref{cor:1} we obtain that \eqref{eq:38}
holds for any $f \in \calC_b(\RR)$. Note that by the reproducing property we have
\begin{equation} \label{eq:41}
	\int_{\RR} K_n(x,y) K_n(y,x) \ud \mu(y) = K_n(x,x).
\end{equation}
Consequently,
\[
	\frac{1}{K_n(x,x)} \int_{\RR} K_n(x,y) K_n(y,x) \ud \mu(y) \equiv 1.
\]
Hence \eqref{eq:39} surely holds if
\begin{equation} \label{eq:40}
	K_n(x,y) K_n(y,x) \geq 0 \quad \text{for a.e. } y \in \supp(\mu), \ n \geq 0.
\end{equation}
However, numerical experiments show that \eqref{eq:40} is \emph{not} satisfied for
the Jacobi-Pi\~{n}eiro polynomials (at least for some choice of parameters). However, it seems that \eqref{eq:39} can still be satisfied. 
So we state the following problem.

\begin{problem} \label{prob:2}
Formulate a criterion which implies that \eqref{eq:39} holds true.
\end{problem}

Let us observe that \eqref{eq:40} together with \eqref{eq:41} implies the positivity of $K_n(x,x)$.

\section{Examples} \label{sec:8}
In this section we shall assume that the sequence of multi-indices from \eqref{eq:24} satisfies
\begin{equation} \label{eq:29a}
	\lim_{\ell \to \infty} 
	\frac{\vec{n}_\ell}{|\vec{n}_\ell|} = 
	(s_1, s_2, \ldots, s_r),
\end{equation}
where
\begin{equation} \label{eq:29b}
	\sum_{i=1}^r s_i = 1 \quad \text{and} \quad s_i > 0.
\end{equation}
This means that the multi-indices $\vec{n}_\ell$ tend to infinity in $\mathbb{N}^r$ in the direction $(s_1,s_2,\ldots,s_r)$.
Next, we shall show that our results can be applied for some well-known systems of multiple
orthogonal polynomials on the real line. 

\subsection{Angelesco systems}
An Angelesco system $\vec{\mu}=(\mu_1, \ldots, \mu_r)$ consists of $r$ measures such that
the convex hull of the support of each measure $\mu_i$ is a compact interval $\Delta_i \subset \RR$
and these intervals are pairwise disjoint. It is a basic fact that
Angelesco systems are perfect provided each $\supp(\mu_i)$ contains infinitely many points
(see, e.g., \cite[Section 23.1.1]{Ismail2009}).

In what follows, we need some concepts from potential theory.
Namely, for positive measures $\eta, \nu$ with \emph{compact} supports, we define their mutual energy by
\begin{equation} \label{eq:30}
	I(\eta,\nu) = \int_\CC \int_\CC \log \frac{1}{|z-w|} \ud \eta(z) \ud \nu(w).
\end{equation}
It was proven in \cite{Gonchar1981} (see also \cite[Chapter 5.6]{Nikishin1991}) that under the assumption that\footnote{For any measure $\mu$ we write $\ud \mu = \mu'(x) \ud x + \ud \mu_{\mathrm{s}}$, where $\mu_{\mathrm{s}}$ is singular with respect to the Lebesgue measure.}
\begin{equation} \label{eq:31}
	\mu'_i(x) > 0, \quad a.e.\ x \in \Delta_i,\ i=1,\ldots,r ,
\end{equation}
there exists a unique minimizer $\vec{\omega} = (\omega_1, \ldots, \omega_r)$ of
\begin{equation} \label{eq:32}
	E(\eta_1, \ldots, \eta_r) = 
	\sum_{j=1}^r I(\eta_j, \eta_j) + 
	\sum_{j=1}^{r-1} \sum_{k=j+1}^r I(\eta_j, \eta_k)
\end{equation}
under the constraint
\begin{equation} \label{eq:33}
	\supp(\eta_i) \subset \Delta_i \quad \text{and} \quad
	\eta_i(\Delta_i) = s_i, \qquad i=1,\ldots,r.
\end{equation}
Moreover, $\nu_n \xrightarrow{w} \nu_\infty$, where $\nu_\infty = \omega_1 + \ldots + \omega_r$.

\begin{theorem}
Let $(\vec{n}_\ell : \ell \geq 0)$ be a path satisfying \eqref{eq:24}, \eqref{eq:29a} and \eqref{eq:29b}.
Let $\vec{\mu}$ be an Angelesco system of measures on the real line satisfying \eqref{eq:31}.
Set $\mu = \mu_1+\ldots+\mu_r$. Then
\[
	\frac{1}{n} K_n(x,x) \ud \mu(x) \xrightarrow{w} \nu_\infty,
\]
where $\nu_\infty = \omega_1 + \ldots + \omega_r$ and $\vec{\omega}$ is the unique minimizer of the problem
\eqref{eq:32} under the constraint \eqref{eq:33}.
\end{theorem}
\begin{proof}
We only need to verify the hypotheses of Corollary~\ref{cor:1}.

First of all, in view of \cite[Remark A.11]{Aptekarev2020} the hypotheses of Proposition~\ref{prop:4} are satisfied.
Hence, the corresponding matrix $J$ satisfies \eqref{eq:NDB}. Next, $\supp(\mu) \subset \bigcup_{i=1}^r \Delta_i$ is compact 
and by \cite[Theorem 23.1.3]{Ismail2009} $\supp(\nu_n) \subset \bigcup_{i=1}^r \Delta_i$. Next, by \cite[Section 4.2]{Kuijlaars2010} one has that \eqref{eq:detPOS} is satisfied. Hence, by Lemma~\ref{lem:2}, the condition \eqref{eq:4} is satisfied. Hence the conclusion follows from Corollary~\ref{cor:1}.
\end{proof}

\subsection{AT systems}
Let $\vec{\mu} = (\mu_1,\ldots,\mu_r)$ be a vector of $r$ measures which are absolutely continuous with respect to a fixed measure $\mu$
on some compact interval $[a,b]$. Let $\ud \mu_i(x) = w_j(x) \ud \mu(x)$. Then $\mu$ is an \emph{algebraic Chebyshev system} (or an AT system for short) if for any multi-index $\vec{n} \in \NN_0^r$ the set
\[
	\Phi_{\vec{n}} = \bigcup_{i=1}^r \bigcup_{k=0}^{n_i - 1} \{ x^k w_i \}
\]
is a Chebyshev system on $[a,b]$, which means that any non-trivial linear combination of 
functions from $\Phi_{\vec{n}}$ has at most $|\vec{n}|-1$ zeros on $[a,b]$. An important result (see \cite[Theorem 23.1.4]{Ismail2009}) states that every AT system is perfect.

\begin{proposition} \label{prop:5}
Let $(\vec{n}_\ell : \ell \geq 0)$ be a path satisfying \eqref{eq:24}, \eqref{eq:29a} and \eqref{eq:29b}.
Let $\vec{\mu}$ be an AT system such that $\ud \mu_i(x) = w_j(x) \ud \mu(x)$ and the functions $w_1, \ldots, w_r$ are continuous. Then the weak limits of the sequences $(\nu_n : n \in \NN_0)$ and $(\eta_n : n \in \NN_0)$ are the same, i.e. \eqref{eq:5} is satisfied.
\end{proposition}
\begin{proof}
We only need to verify the hypotheses of Corollary~\ref{cor:1}.

First of all, in view of \cite[Theorem 23.1.4]{Ismail2009} and \cite[Theorem 2.1]{Haneczok2012} the hypotheses of Theorem~\ref{thm:3} are satisfied. Hence, the corresponding matrix $J$ satisfies \eqref{eq:NDB}. 
Next, by \cite[Section 4.3]{Kuijlaars2010} one has that \eqref{eq:detPOS} is satisfied. Hence, by Lemma~\ref{lem:2}, the condition \eqref{eq:4} is satisfied. Hence the conclusion follows from Corollary~\ref{cor:1}. 
\end{proof}

Proposition~\ref{prop:5} can be applied to Jacobi-Pi\~{n}eiro polynomials (see \cite[Chapter 23.3.2]{Ismail2009} for more details). The description of the limiting distribution of $(\nu_n : n \in \NN)$ for a very special choice of the path $(\vec{n}_\ell : \ell \geq 0)$ has been considered in \cite{Neuschel2016, Coussement2008}.

\subsubsection{Nikishin systems}
Let $\vec{\sigma} = (\sigma_1, \ldots, \sigma_r)$ be a vector of finite measures such that the convex hull of of support of each measure $\sigma_i$ is a compact interval $\Delta_i \subset \RR$ and 
\[
	\Delta_j \cap \Delta_{j+1} = \emptyset, \quad j=1,\ldots,r-1.
\]
For two measures $\eta, \nu$ let us define a measure $\langle \eta, \mu \rangle$ by
\begin{equation} \label{eq:36}
	\ud \langle	\eta, \mu \rangle (x) = \calC(\mu)(x) \ud \eta(x),
\end{equation}
where for any measure $\mu$ on the real line its Cauchy transform is defined by
\begin{equation} \label{eq:35}
	\calC(\mu)(z) = \int_\RR \frac{\ud \mu(y)}{z-y}, \quad z \notin \supp(\mu).
\end{equation}
Then we say that $\vec{\mu} = (\mu_1,\ldots,\mu_r)$ is a Nikishin system generated by $\vec{\sigma}$ if
\begin{equation} \label{eq:34}
	\mu_1 = \sigma_1, \quad 
	\mu_2 = \langle \sigma_1, \sigma_2 \rangle, \quad \ldots \quad
	\mu_r = \big\langle \sigma_1, \langle \sigma_2, \ldots, \sigma_r \rangle \big\rangle.
\end{equation}
It has been proven in \cite{Fidalgo2011} that any Nikishin system is an AT system.

It was proven in \cite[Chapter 5.7]{Nikishin1991} that under the assumption that
\begin{equation} \label{eq:31'}
	\sigma'_i(x) > 0, \quad a.e.\ x \in \Delta_i,\ i=1,\ldots,r ,
\end{equation}
there exists a unique minimizer $\vec{\omega} = (\omega_1, \ldots, \omega_r)$ of
\begin{equation} \label{eq:32'}
	E(\eta_1, \ldots, \eta_r) = 
	\sum_{j=1}^r I(\eta_j, \eta_j) - 
	\sum_{j=1}^{r-1} I(\eta_j, \eta_{j+1})
\end{equation}
under the constraint
\begin{equation} \label{eq:33'}
	\supp(\eta_i) \subset \Delta_i \quad \text{and} \quad
	\eta_i(\Delta_i) = \sum_{j=i}^r s_j, \qquad i=1,\ldots,r.
\end{equation}
Moreover, $\nu_n \xrightarrow{w} \nu_\infty$, where $\nu_\infty = \omega_1$.

\begin{theorem}
Let $(\vec{n}_\ell : \ell \geq 0)$ be a path satisfying \eqref{eq:24}, \eqref{eq:29a} and \eqref{eq:29b}.
Let $\vec{\mu}$ be a Nikishin system generated by $\vec{\sigma}$ satisfying \eqref{eq:31'}.
Set $\mu = \mu_1$. Then
\[
	\frac{1}{n} K_n(x,x) \ud \mu(x) \xrightarrow{w} \nu_\infty,
\]
where $\nu_\infty = \omega_1$ and $\vec{\omega}$ is the unique minimizer of the problem
\eqref{eq:32'} under the constraint \eqref{eq:33'}.
\end{theorem}
\begin{proof}
Observe that by \eqref{eq:34}, \eqref{eq:36} and \eqref{eq:35} the densities of $\mu_i$ with respect to $\mu_1$
are continuous functions. Moreover, by the discussion above $\nu_n \xrightarrow{w} \nu_\infty$.
Hence, by Proposition~\ref{prop:5} the result follows.
\end{proof}

Let us mention that in \cite{Gonchar1997} (see also \cite{Aptekarev2010}) a generalisation of both Angelesco and Nikishin systems has been proposed. For such systems it is known that $\nu_n \xrightarrow{w} \nu_\infty$ for some measure $\nu_\infty$ which also comes from a minimalization problem. It is quite possible that our Corollary~\ref{cor:1} can be applied also here.

\subsection*{Acknowledgment}
The authors would like to thank Guilherme Silva for informing us about \cite{Hardy2015} and anonymous referees for useful suggestions.
The first author was supported by the Methusalem grant \textit{Classification, symmetries and singularities at the frontiers of algebra, analysis and geometry} of the Flemish Government. The second author was supported by FWO grant G0C9819N of the Research Foundation -- Flanders.

\begin{bibliography}{CD-MOP.bib}
	\bibliographystyle{amsplain}
\end{bibliography}
\end{document}